\documentclass{amsart}
\usepackage{latexsym}
\usepackage{amsmath}
\usepackage{amssymb}
\usepackage{amsthm}
\usepackage{amscd}
\usepackage{enumerate} 
\usepackage{amssymb} 
\usepackage{mathrsfs}     
\usepackage[all]{xy}
\usepackage{color}
\usepackage{graphicx}
\usepackage{mathtools}
\usepackage[dvipdfmx]{hyperref}
\hypersetup{colorlinks=true}

\title{Bogomolov-Sommese type vanishing for globally $F$-regular threefolds}
\author{Tatsuro Kawakami}
\address{Graduate School of Mathematical Sciences, University of Tokyo, 3-8-1 Komaba, Meguro-ku, Tokyo 153-8914, Japan}
\email{kawakami@ms.u-tokyo.ac.jp}

\def\phi{\varphi}
\def\epsilon{\varepsilon}

\def\log{\operatorname{log}}
\def\Hom{\operatorname{Hom}}

\def\Supp{\operatorname{Supp}}
\def\codim{\operatorname{codim}}
\def\Pic{\operatorname{Pic}}

\def\Im{\operatorname{Im}}
\def\Ker{\operatorname{Ker}}

\def\Ex{\operatorname{Exc}}
\def\rank{\operatorname{rank}}

\def\sg{\operatorname{sg}}
\def\reg{\operatorname{reg}}
\def\snc{\operatorname{snc}}
\def\nsnc{\operatorname{nsnc}}
\def\max{\operatorname{max}}
\def\NS{\operatorname{NS}}

\newcommand{\F}{\mathcal{F}}

\newcommand{\Q}{\mathbb{Q}} 
\newcommand{\C}{\mathbb{C}} 
 
\newcommand{\Z}{\mathbb{Z}}
\newcommand{\PP}{\mathbb{P}}

\newcommand{\sO}{\mathcal{O}}

\theoremstyle{plain}
\newtheorem{thm}{Theorem}[section] 
\newtheorem{cor}[thm]{Corollary}
\newtheorem{prop}[thm]{Proposition}

\newtheorem{lem}[thm]{Lemma}
\theoremstyle{definition} 
\newtheorem{defn}[thm]{Definition}

\newtheorem{eg}[thm]{Example} 

\theoremstyle{remark}
\newtheorem{rem}[thm]{Remark}
\newtheorem{ques}[thm]{Question}

\newtheorem*{notation}{Notation} 
\newtheorem*{cl}{Claim}

\begin{document}
\tolerance = 9999

\maketitle
\markboth{TATSURO KAWAKAMI}{Bogomolov-Sommese type vanishing for globally $F$-regular threefolds}

\begin{abstract}
In this paper, we show that every invertible subsheaf of the cotangent bundle of a smooth globally $F$-regular threefold of characteristic $p>3$ has Iitaka dimension less than or equal to one. 
\end{abstract}

%%%%%%%%%%%%%%%%%%%%%%%%%%%%%%%%%%%%%%%%%%%%%%%%%%%%%%%%%%%%%%%%%%%%%%
%%%%%%%%%%%%%%%%%%%%%%%%%%%%%%%%%%%%%%%%%%%%%%%%%%%%%%%%%%%%%%%%%%%%%%
\section{Introduction}
%%%%%%%%%%%%%%%%%%%%%%%%%%%%%%%%%%%%%%%%%%%%%%%%%%%%%%%%%%%%%%%%%%%%%%
%%%%%%%%%%%%%%%%%%%%%%%%%%%%%%%%%%%%%%%%%%%%%%%%%%%%%%%%%%%%%%%%%%%%%%
Differential sheaves are vector bundles naturally attached to smooth algebraic varieties and it is important to study their positivity properties.
The following theorem describes the positivity of line bundles contained in the differential sheaves.

\begin{thm}[\textup{Bogomolov-Sommese vanishing theorem, \cite[Corollary 1.3]{Gra15}}]\label{BV}
Let $(X, \Delta)$ be a projective log canonical pair over the field of complex numbers $\C$ and $D$ a Weil divisor on $X$.
If $\sO_X(D)$ is a subsheaf of $\Omega_X^{[i]}(\log \Delta)$, then $\kappa(X, D) \leq i$. 
\end{thm}

Here, $\Omega_X^{[i]}(\log \Delta)$ denotes the reflexive differential form and $\kappa(X, D)$ denotes the Iitaka dimension of $D$ (see \cite[2.D.]{GKKP} for the details).
Theorem \ref{BV} is called a vanishing theorem because it is equivalent to saying that
\[
H^0(X, (\Omega^{[i]}_X(\log \Delta)\otimes \sO_X(-D))^{**})=0 
\]
for every Weil divisor $D$ with $\kappa(X, D)>i$, where $(-)^{**}$ is double dual.  Theorem \ref{BV} was proved by Bogomolov \cite{Bog} when $X$ is smooth and $\Delta=0$, and is generalized by Graf, Greb, Jabbusch, Kebekus, Kov\'acs, Peternell, Sommese, etc (\cite{Gra15}, \cite{GKK}, \cite{GKKP}, \cite{JK}, \cite{SS85}).
Bogomolov used Theorem \ref{BV} to prove the inequality $c_1^2 \leq 4c_2$ for smooth projective surfaces of general type. This was improved by Miyaoka \cite{Miy77} to $c_1^2 \leq 3c_2$, but Theorem \ref{BV} was still used in this proof.

In this paper, we discuss what happens if the variety is defined over a field of positive characteristic.  
In general, Theorem \ref{BV} fails in positive characteristic.  
There exists a smooth projective surface of general type which lifts to the ring of Witt vectors (and therefore satisfies the Kodaira vanishing theorem), but its cotangent bundle contains a big invertible sheaf (see \cite[Example 1]{Lan15a}). If $\Delta\neq0$, then Theorem \ref{BV} fails even if $X$ is a smooth projective rational surface (see \cite[Lemma 8.3]{Lan}, \cite[Proposition 11.1]{Langer19}).
However, Mehta-Ramanathan \cite{mehta--ramanathan} proved that Kodaira type vanishing theorems hold for smooth projective $F$-split varieties,
an important class of algebraic varieties defined in terms of Frobenius splitting.  
Then it is natural to ask the following question.
\begin{ques}\label{BVP}
Does the Bogomolov-Sommese vanishing hold for smooth projective $F$-split varieties?
\end{ques}

We can show an analog of Theorem \ref{BV} holds for smooth $F$-split surfaces when $\Delta=0$ (see Theorem \ref{BVonFSS}),
but the proof heavily depends on the classification of smooth projective surfaces and does not work for higher-dimensional varieties.
In higher-dimensional cases, we consider globally $F$-regular varieties, a special class of $F$-split varieties, introduced by Smith \cite{Smi}
(see Definition \ref{defi-gFreg} (2) for the definition). For example,  smooth $F$-spilt Fano varieties are globally $F$-regular. We give a partial affirmative answer to Question \ref{BVP} when $X$ is a smooth globally $F$-regular threefold. 

\begin{thm}[Theorem \ref{main}]\label{mainintro}
Let $X$ be a smooth projective globally $F$-regular threefold over a perfect field of characteristic $p>3$. 
Let $\sO_X(D) \subset \Omega_X$ is an invertible subsheaf. 
Then $\kappa(X, D) \leq 1$.
Furthermore, if $p>7$, then we have $\kappa(X, D) \leq 0$.
\end{thm}

We need the assumption ``$p>3$'' only for running the minimal model program (MMP, for short), which was recently established for threefolds of characteristic $p>3$. (see \cite{Birker}, \cite{BW17}, \cite{HW19}, and \cite{HX15} for the details). 
In order to prove Theorem \ref{mainintro}, we show that $H^0(X, \Omega_X\otimes \sO_X(-D))=0$ for every Cartier divisor with $\kappa(X, D)\geq2$. We first consider the case where $D$ is nef and big.

\begin{thm}[Theorem \ref{key}]\label{IntroKey}
Let $X$ be a projective globally $F$-regular variety over a perfect field of characteristic $p>0$ and $\Delta$ a Weil divisor on $X$.
Suppose that $\dim X \geq 2$ and the non-simple normal crossing locus of $(X, \Delta)$ has codimension at least $3$. 
Then $H^0(X, (\Omega_X^{[1]}(\log \, \Delta)\otimes \sO_X(-D))^{**})=0$ for every nef and big $\Q$-Cartier Weil divisor $D$ on $X$.
\end{thm}

In Theorem \ref{IntroKey}, we use the global $F$-regularity of $X$ to singular varieties.
In the proof of Theorem \ref{mainintro}, we run the MMP to make $D$ satisfy the assumption of Theorem \ref{IntroKey}.
In general, even if we start from a smooth variety, an output of the MMP is not necessarily smooth.
This is the reason why we have to consider singular varieties in Theorem \ref{IntroKey}.

\begin{notation}
Throughout this paper, we work over a perfect field $k$ of characteristic $p>0$.
A \textit{variety} means an integral separated scheme of finite type over $k$. 
A \textit{curve} (resp.~\textit{surface}, \textit{threefold}) means
a variety of dimension one (resp.~two, three).
\end{notation}

%%%%%%%%%%%%%%%%%%%%%%%%%%%%%%%%%%%%%%%%%%%%%%%%%%%%%%%%%%%%%%%%%%%%%%
%%%%%%%%%%%%%%%%%%%%%%%%%%%%%%%%%%%%%%%%%%%%%%%%%%%%%%%%%%%%%%%%%%%%%%
\section{Preliminaries}
%%%%%%%%%%%%%%%%%%%%%%%%%%%%%%%%%%%%%%%%%%%%%%%%%%%%%%%%%%%%%%%%%%%%%%
%%%%%%%%%%%%%%%%%%%%%%%%%%%%%%%%%%%%%%%%%%%%%%%%%%%%%%%%%%%%%%%%%%%%%%
\subsection{Reflexive sheaves and birational maps}
Let $X$ be a normal variety. For any coherent sheaf $\F$ on $X$, we denote by $\F^{**}$ the \textit{double dual} of $\F$, that is, ${\F}^{**}\coloneqq\mathcal{H}om_{\sO_X}(\mathcal{H}om_{\sO_X}(\F, \sO_X), \sO_X)$. The sheaf $\F$ is called \emph{reflexive} if the natural map $\F \to \F^{**}$ is an isomorphism. 
A \emph{Weil divisorial sheaf} is a reflexive sheaf of rank one. We recall that there is one-to-one corresponding between a Weil divisor $D$ on X and a Weil divisorial sheaf $\sO_X(D)$.\\
We say that $(X, \Delta)$ is a \emph{pair} if $X$ is a normal variety and $\Delta$ is an effective Weil divisor on $X$. We say that a pair $(X, D)$ is \textit{reduced} if $D$ is reduced.
We denote by $X_{\reg}$ (resp.~$X_{\sg}$) the \emph{regular locus} (resp.~\emph{singular locus}) of $X$. 
We denote by $(X, \Delta)_{\snc}$ (resp.~$(X, \Delta)_{\nsnc}$) the \emph{simple normal crossing locus} (resp.~\emph{non-simple normal crossing locus}) of a pair $(X, \Delta)$.
For a pair $(X, \Delta)$ satisfying $\codim_X((X, \Delta)_{\nsnc})\geq 2$, we define the sheaf of the \emph{reflexive differential $j$-form} as
$\Omega_X^{[j]}(\log\, \Delta) \coloneqq i_{*}(\Omega_U^j(\log\, \Delta))$, 
where $i\colon U\coloneqq (X, \Delta)_{\snc}\hookrightarrow X$ is the natural inclusion map
and $\Omega_U^j(\log \, \Delta)$ is the sheaf of the logarithmic K\"ahler differentials.

Let us recall the definition of the Iitaka dimension for Weil divisors. 
\begin{defn}[\textup{\cite[Definition 2.18]{GKKP}}]
Let $X$ be a normal projective variety and $D$ a Weil divisor on $X$.  
We define the \emph{Iitaka dimension}$\in\{0,1,\cdots,\dim X\}$ as follows.
If $h^0(X, \sO_X(mD)) = 0$ for all $m \in \Z_{>0}$, we say $D$ has Iitaka dimension 
$\kappa(X, D) \coloneqq -\infty$.  Otherwise, set
\[ M \coloneqq \bigl\{ m\in \Z_{>0} \;\big|\; h^0(X, \sO_X(mD)) > 0 \bigr\}, \]
and consider the natural rational mappings
\[ \phi_m \colon X \dasharrow \mathbb P\bigl(H^0(X, \sO_X(mD))^*\bigr) \quad \text{ for each } m \in M. \]
Note that we can consider the rational map as above since $\sO_X(D)$ is invertible on $X_{\reg}$. 
The Iitaka dimension of $D$ is then defined as
\[ \kappa(X, D) \coloneqq \max_{m \in M} \bigl\{ \dim \overline{\phi_m(X)} \bigr\}. \]
We say $D$ is \emph{big} if $\kappa(X, D) = \dim X$.
We note that when $D$ is ($\Q$-)Cartier, the above definition coincides with that of the Iitaka dimension for ($\Q$-)Cartier divisors (see \cite[Definition 2.13]{Lazarsfeld} for example).
 \end{defn}

Let $f \colon X \to Y$ be a birational morphism of normal varieties. 
For a prime divisor $F$ on $X$, 
the {\em push-forward} $f_*F$ of $F$ by $f$ (or to $Y$) is defined as follows. 
If $F \subset \Ex(f)$, then $f_*F=0$, and if $F \not\subset \Ex(f)$, 
then $f_*F$ is defined as the prime divisor whose generic point equals to the generic point of $f(F)$. 
For a $\Q$-Weil divisor $D$ on $X$, 
the {\em push-forward} $f_*D$ of $D$ by $f$ 
is defined as $\sum_{i \in I} r_i f_*D_i$, where $D=\sum_{i \in I} r_i D_i$ is the irreducible decomposition of $D$. 
Next, let $f \colon X \dashrightarrow Y$ be a birational map of normal varieties. 
For a $\Q$-Weil divisor $D$ on $X$,
the {\em push-forward} $f_*D$ of $D$ by $f$ is defined as $(f|_U)_*D|_{U}$,
where $U$ denotes the maximum open subset of $X$ such that $f$ is defined.
Let $f \colon X \dasharrow Y$ be a birational map of normal varieties.
We say $f$ is a {\em birational contraction} if $f^{-1}$ does not contract any divisor, that is, there is no divisor $D$ on $Y$ such that $f^{-1}_{*}D=0$.
We note that both divisorial contractions and flips, which appear in a sequence of the MMP, are birational contractions. 

The following lemma shows that a Bogomolov-Sommese type vanishing can be reduced to an output of the MMP.

\begin{lem}\label{push}
Let $f \colon X\dasharrow X'$ be a birational contraction of normal projective varieties and $D$ a Weil divisor on $X$.
Let $D'$ be the push-forward of $D$ by $f$. Then the following hold.
\begin{enumerate}
    \item[\textup{(1)}] $\kappa(X, D)\leq \kappa(X',D')$. 
    \item[\textup{(2)}] There exists an injective map $H^0(X, (\Omega_X^{[i]}\otimes \sO_X(-D))^{**})\hookrightarrow H^0(X', (\Omega_{X'}^{[i]}\otimes \sO_{X'}(-D'))^{**})$. 
\end{enumerate}
\end{lem}
\begin{proof}
We first show the assertion (2).
Since $f^{-1}$ does not contract any divisor, there exists an open subset $V\subset X'$ with $\codim_{X'}(X'-V)\geq 2$ such that $f|_{f^{-1}(V)} \colon f^{-1}(V) \cong V$ is an isomorphism and $V$ is contained in $X'_{\reg}$. 
Then we have 
\[
\begin{array}{rl}
H^0(X, (\Omega^{[i]}_X\otimes \sO_X(-D))^{**})\subset&H^0(f^{-1}(V), \Omega^{i}_{f^{-1}(V)}\otimes \sO_{f^{-1}(V)}(-D))\\
                                                          =&H^0(V, \Omega^{i}_{V}\otimes \sO_{V}(-D'))\\
                                                          =&H^0(X', (\Omega^{[i]}_{X'}\otimes \sO_{X'}(-D'))^{**}),\\
                                                              \end{array}
\]
and hence we obtain the assertion (2).
An argument similar to the above shows that $H^0(X, \sO_X(mD))\subset H^0(X', \sO_{X'}(mD'))$ for all $m\in \Z_{>0}$ and hence we obtain the assertion (1).
\end{proof}

%%%%%%%%%%%%%%%%%%%%%%%%%%%%%%%%%%%%%%%%%%%%%%%%%%%%%%%%%%%%%%%%%%%%%%%%%%%%%%%%%%%%%%%%%%%%%%%%%%%%%%%%%%%%%%%%%%%%%%%%%%%%%%%%%%%%%%%%%%%%%%%%%%%%%%%%%%%%%%%%%%%%%%
\subsection{$F$-split and globally $F$-regular varieties}
In this subsection, we gather the results about $F$-split and globally $F$-regular varieties.

\begin{defn}[\textup{\cite{mehta--ramanathan}, \cite{Smi}}]\label{defi-gFreg} 
Let $X$ be a normal variety. 
\begin{enumerate}
\item[\textup{(1)}] We say that $X$ is \textit{(globally) $F$-split} if the Frobenius map $\mathcal{O}_X \to F_{*}\mathcal{O}_X$ splits as an $\mathcal{O}_X$-module homomorphism. 
\item[\textup{(2)}] We say that $X$ is \textit{globally $F$-regular} if for every effective Weil divisor $D$ on $X$, there exists an integer $e \geq 1$ such that the composite map 
\[
\mathcal{O}_X \to F^e_*\mathcal{O}_X \hookrightarrow F^e_*\mathcal{O}_X(D)
\]
of the $e$-times iterated Frobenius map $\mathcal{O}_X \to F^e_*\mathcal{O}_X$ and the natural inclusion $F^e_*\mathcal{O}_X \hookrightarrow F^e_*\mathcal{O}_X(D)$ 
splits as an $\mathcal{O}_X$-module homomorphism. 
\end{enumerate}
\end{defn}

\begin{rem}\label{gFreg}
\begin{enumerate}
\item[\textup{(1)}] Let $X$ be a globally $F$-regular variety and $D$ an effective Weil divisor on $X$.
Then the map
$\mathcal{O}_X  \hookrightarrow F^e_*\mathcal{O}_X(D)$
splits as an $\mathcal{O}_X$-module homomorphism for all sufficiently large $e$ by \cite[Proposition 3.8]{SS10}.

\item[\textup{(2)}] Let $f\colon X \dasharrow Y$ be a small birational map or a projective surjective morphism satisfying $f_{*}\sO_X=\sO_Y$ of normal varieties.  
If $X$ is $F$-split (resp.~globally $F$-regular), then so is $Y$ by \cite[Lemma 1.5]{GLP$^+$15}. In particular, if we start the MMP from an $F$-split (resp.~globally $F$-regular) variety, then an output of the MMP is also $F$-split (resp.~globally $F$-regular).

\item[\textup{(3)}] Globally $F$-regular varieties are Cohen-Macaulay by \cite[Proposition 4.1]{Smi}.
\end{enumerate}
\end{rem}

We need the following theorems in Section \ref{Sec:GFR}.

\begin{thm}[\textup{\cite[Theorem 2.1]{GLP$^+$15}}]\label{fiber}
Let $f \colon X \to Y$ be a projective surjective morphism of normal varieties satisfying $f_{*}\sO_X = \sO_Y$.
If $X$ is globally $F$-regular, then a general fiber of $f$ is normal and globally $F$-regular.
\end{thm}

\begin{thm}[\textup{Proof of \cite[Theorem 4.1]{GLP$^+$15}}]\label{glsrc}
Let $f \colon X\to Y$ be a projective surjective morphism from a terminal globally $F$-regular threefold to a normal variety over an algebraically closed field of characteristic $p>0$ satisfying $f_*\mathcal{O}_X=\mathcal{O}_Y$.
Suppose that $-K_X$ is $f$-ample and one of the following conditions holds.
\begin{enumerate}
    \item[\textup{(1)}] $\dim Y=2$.
    \item[\textup{(2)}] $p>7$ and $\dim Y=1$.
\end{enumerate}
Then $X$ is separably rationally connected.
\end{thm}

We refer to \cite[IV 3.2 Definition]{Kol96} for the definition of separably rationally connected varieties.
Since the separably rationally connected property is preserved under birational maps, Theorem \ref{glsrc} states that if we start $K_{X}$-MMP from a smooth globally $F$-regular threefold and the MMP ends up with a Mori fiber space over a surface, or a curve and $p>7$, then $X$ is separably rationally connected. 
On the other hand, very little is known when the MMP ends up with a Fano variety. We refer to \cite{GLP$^+$15} for more details.

%%%%%%%%%%%%%%%%%%%%%%%%%%%%%%%%%%%%%%%%%%%%%%%%%%%%%%%%%%%%%%%%%%%%%%%%%%%%%%%
\subsection{Logarithmic Cartier operators}
In this subsection, we recall the logarithmic Cartier operator. Let $X$ be a smooth variety and $\Delta$ a simple normal crossing divisor on $X$.
The Frobenius push-forward of the logarithmic de Rham complex
\[
F_{*}\Omega^{\bullet}_X(\log \Delta) \colon F_{*}\sO_X \overset{F_{*}d}{\to} F_{*}\Omega^1_X(\log \Delta) \overset{F_{*}d}{\to} \cdots
\]
is a complex of $\sO_X$-module homomorphisms.
For all $i\geq0$, we define locally free $\sO_X$-modules as follows.
\[
\begin{array}{rl}
&B_X^i(\log \, \Delta)\coloneqq\Im(F_{*}d \colon F_{*}\Omega^{i-1}_X(\log \,\Delta) \to F_{*}\Omega^i_X(\log \, \Delta)),\\
&Z_X^i(\log \, \Delta)\coloneqq\Ker(F_{*}d \colon F_{*}\Omega^i_X(\log \, \Delta) \to F_{*}\Omega^{i+1}_X(\log \,\Delta)).\\
\end{array}
\]
By definition, we have the following exact sequence
\[
0 \to Z_X^i(\log \, \Delta) \to F_{*}\Omega^i_X(\log \, \Delta) \to B_X^{i+1}(\log \, \Delta) \to 0
\]
for $i\geq 0$. 
 
In particular, by taking $i=0$ in the above exact sequence, we have 
\[
0 \to \sO_X \to F_{*}\sO_X \overset{F_{*}d}{\to} B^1_X\to 0.
\]
We note that $B_X^1(\log \, \Delta)=B_X^1$.
We remark that the $F$-splitting (Definition \ref{defi-gFreg} (1)) is nothing but to the splitting of this exact sequence. 
Moreover, we have the exact sequence arising from the logarithmic Cartier isomorphism
\[
0 \to B_X^i(\log \, \Delta) \to  Z_X^i(\log \, \Delta) \overset{C}{\to} \Omega^i_X(\log \, \Delta) \to 0.
\]
We refer to \cite[Theorem 7.2]{Kat70} for more details.

%%%%%%%%%%%%%%%%%%%%%%%%%%%%%%%%%%%%%%%%%%%%%%%%%%%%%%%%%%%%%%%%%%%%%%%%%%%%%%%%
%%%%%%%%%%%%%%%%%%%%%%%%%%%%%%%%%%%%%%%%%%%%%%%%%%%%%%%%%%%%%%%%%%%%%%%%%%%%%%%%
%%%%%%%%%%%%%%%%%%%%%%%%%%%%%%%%%%%%%%%%%%%%%%%%%%%%%%%%%%%%%%%%%%%%%%%%%%%%%%%%
%%%%%%%%%%%%%%%%%%%%%%%%%%%%%%%%%%%%%%%%%%%%%%%%%%%%%%%%%%%%%%%%%%%%%%%%%%%%%%%%
\section{Bogomolov-Sommese type vanishing for several varieties}
In this section, we discuss a Bogomolov-Sommese type vanishing on varieties with special properties. 
By using these results and the classification of smooth projective surfaces, we show a Bogomolov-Sommese type vanishing for smooth projective $F$-split surfaces.

The following assertion states about a Bogomolov-Sommese type vanishing on separably uniruled varieties.
We refer to \cite[IV 1.1 Definition]{Kol96} for the definition of separably uniruled varieties.

\begin{prop}[\textup{\cite[Lemma 7]{Kol}}]\label{BVonSU}
Let $X$ be a smooth projective separably uniruled variety and $D$ a big Cartier divisor on $X$.
Then $H^0(X, \Omega^i_X\otimes \sO_X(-D))=0$ for all $i\geq 0$. 
\end{prop}

\begin{rem}\label{BVonSur}
Let $X$ be a smooth projective surface and $\Delta$ a simple normal crossing divisor. 
If $\kappa(X, K_X)=-\infty$ and $\Delta=0$, then the Bogomolov-Sommese vanishing (Theorem \ref{BV}) holds by Proposition \ref{BVonSU}.
However, there exists a pair $(S, E)$ consisting of a smooth rational surface $S$ and a disjoint union of smooth rational curves $E$ such that the logarithmic differential form $\Omega_S(\log \, E)$ contain a big invertible sheaf in every characteristic $p>0$ (see \cite[Lemma 8.3]{Lan} or \cite[Proposition 11.1]{Langer19}). 
On the other hand, we will see that if $X$ is $F$-split, then $\Omega_X(\log \, \Delta)$ dose not contain a nef and big invertible sheaf in Proposition \ref{ANvani}.
\end{rem}

Next, we show a Bogomolov-Sommese type vanishing on separably rationally connected varieties.
Let $X$ be a smooth projective variety of $\dim\,X=n$. We recall that a rational curve $\phi \colon \PP_k^1 \to X$ is called \textit{very free} if $\phi^{*}\Omega_X=\sO_{\PP_k^1}(-a_1)\oplus \cdots \oplus\sO_{\PP_k^1}(-a_n)$ 
for $a_1, \cdots, a_n\in \Z_{>0}$.

\begin{thm}\label{SRC}
Let $X$ be a smooth projective variety over an algebraically closed field. Then $X$ is separably rationally connected if and only if there is a very free rational curve through a general point of $X$.
\end{thm}
\begin{proof}
We refer to \cite[IV. Theorem 3.7]{Kol96} for the proof.
\end{proof}

The proof of the following proposition is essentially same as that of Proposition \ref{BVonSU}, but we include the proof for the convenience of the reader.  

\begin{prop}\label{BVonSRC}
Let $X$ be a smooth projective separably rationally connected variety
and $D$ a Cartier divisor on $X$. 
If $\kappa(X, D)\geq 0$,
then $H^0(X, \Omega_X^i \otimes \sO_X(-D))=0$ for all $i>0$.
\end{prop}
\begin{proof}
By replacing $k$ with its algebraic closure, we may assume that $k$ is an algebraically closed field.
We take a Cartier divisor $D$ satisfying $\kappa(X, D)\geq 0$.
Conversely, we assume that there exists a nonzero section $0\neq s\in H^0(X, \Omega_X^i\otimes \sO_X(-D))$ for some $i>0$.
We fix $m\in \Z_{>0}$ such that $mD$ is linearly equivalent to an effective divisor.
We take a very free rational curve $\phi \colon \PP_k^1 \to X$ through a general point of $X$.
Then $\Im \phi$ is not contained $\Supp(mD)$ and hence we have $\phi^{*}\sO_X(D)=\sO_{\PP_k^1}(b)$ for some $b\geq0$.
By the definition of a very free rational curve, it follows that $\phi^{*}(\Omega_X^i\otimes \sO_X(-D))=\sO_{\PP_k^1}(-b_1)\oplus \cdots \oplus\sO_{\PP_k^1}(-b_n)$ for some $b_1, \cdots, b_n>0$.

On the other hand, since $\phi \colon \PP_k^1 \to X$ passes through a general point of $X$, it follows that $\Im \phi$ is not contained in the zero locus of $s$ and hence $s|_{\Im \phi}\neq 0$.

Now, we obtain 
\[
\begin{array}{rl}
0\neq s|_{\Im \phi}\in& H^0(\Im \phi, (\Omega_{X}^i\otimes \sO_X(-D))|_{\Im \phi})\\
\hookrightarrow&H^0(\Im \phi,((\Omega_{X}^i \otimes \sO_X(-D))\otimes \phi_{*}\sO_{\PP_k^1})\\
=&H^0(\PP_k^1,\phi^{*}((\Omega_{X}^i \otimes \sO_{X}(-D))\\
=&H^0(\PP_k^1,\sO_{\PP_k^1}(-b_1)\oplus \cdots \oplus\sO_{\PP_k^1}(-b_n))\\
=&0,
\end{array}
\]
a contradiction.
\end{proof}

\begin{rem}\label{Rem:BVonGFRS}
Let $X$ be a smooth globally $F$-regular surface. 
Then $X$ is rational by \cite[Proposition 3.5]{GLP$^+$15}. Therefore
if $\sO_X(D)\subset \Omega^i_X$ is an invertible subsheaf for some $i>0$, then $\kappa(X, D)=-\infty$ by Proposition \ref{BVonSRC}.
\end{rem}

\begin{lem}\label{ruled}
Let $f \colon X \to C$ be a minimal ruled surface and $g$ denote $\dim_{k} H^1(X, \sO_X)$.
Then $H^0(X, \Omega_X\otimes \sO_X(-D))=0$ for every Cartier divisor $D$ satisfying $\kappa(X, D)\geq\min\{g, 2\}$.
\end{lem}
\begin{proof}
If $g=0$ or $g>1$, then the assertion follows from Proposition \ref{BVonSRC} and Proposition \ref{BVonSU}, respectively.
We assume that $g=1$.
We take an injective homomorphism $s \colon \sO_X(D)\hookrightarrow \Omega_X$.
Then we have the following commutative diagram
\begin{equation*}
\xymatrix{ & & \sO_X(D) \ar@{.>}[ld] \ar[d]^{s} \ar[rd]^{t} &\\
                 0\ar[r] &f^{*}\omega_C=\sO_X \ar[r]   & \Omega^1_X \ar[r]  & \omega_X \ar[r] & 0.}
\end{equation*}
Let $F$ be a general fiber of $f$. By the generality of $F$, the restriction $t|_{F}\colon \sO_{F}(D)\hookrightarrow \omega_X|_F$ is injective.
Then we have $(D \cdot F)\leq (K_X \cdot F)=-2$ and hence $\kappa(X, D)=-\infty$ by the nefness of $F$.
Now, we assume that $t$ is zero. Then an injective homomorphism $\sO_X(D) \hookrightarrow \sO_X$ is induced by the above diagram and hence  $\kappa(X, D)=-\infty$ or $D=0$.
Therefore the assertion holds.
\end{proof}

%%%%%%%%%%%%%%%%%%%%%%%%%%%
%%%%%%%%%%%%%%%%%%%%%%%%%%%
Next, we discuss the case where $\kappa(X, K_X)=0$.
The following proposition is an immediate consequence of \cite[Corollary 3.3]{Lan15b}.

\begin{prop}\label{not-uniruled}
Let $X$ be a smooth projective variety of $\dim\,X=n$. 
Suppose that $p\geq (n-1)(n-2)$, $K_X\equiv0$, and $X$ is not uniruled.
If $\sO_X(D)\subset \Omega_X^{i}$ is an invertible subsheaf for some $i\geq 0$,
then $\kappa(X, D)\leq 0$.
\end{prop}
\begin{proof}
We may assume that $k$ is an algebraically closed field. 
By \cite[Corollary 3.3]{Lan15b}, it follows that $\Omega_X$ is strongly semistable with respect to any ample polarization $H$ and so is $\Omega^i_X$ for each $i\geq 0$ by \cite[Theorem 3.23]{RR84}. 
We refer to \cite{Lan15b} for the definition of the strongly semistability.
We take an injective homomorphism $\sO_X(D)\hookrightarrow \Omega_X^i$ for some $i\geq0$. 
By the definition of the semistability, we have $(D\cdot H^{n-1})\leq (-c_1(X)\cdot H^{n-1})/\rank\,\Omega_X^i=0$ and hence $\kappa(X, D)\leq 0$. 
Therefore we obtain the assertion.
\end{proof}

\begin{rem}\label{rem-not-uniruled}
A Calabi-Yau variety whose Artin-Mazur height is finite is not uniruled by \cite[Theorem 1.3]{Hir99}. 
We note that the proof of \cite[Theorem 1.3]{Hir99} works in any dimension.
\end{rem}

Now, we show a Bogomolov-Sommese type vanishing for smooth projective $F$-split surfaces.

\begin{thm}\label{BVonFSS}
Let $X$ be a smooth projective $F$-split surface.
If $\sO_X(D)\subset \Omega_X^{i}$ is an invertible subsheaf for some $i\geq 0$,
then $\kappa(X, D)\leq 0$.
\end{thm}
\begin{proof}
When $i=0$, the assertion is obvious.
Since $X$ is $F$-split, 
there exists a section $\sigma \in  H^0(X, \sO_X((1-p)K_X))\cong\Hom_{\sO_X}(F_{*}\sO_X, \sO_X)$ which induces a splitting of the Frobenius map $\sO_X\to F_{*}\sO_X$, and in particular
the anti-canonical $-K_X$ is effective.
Then the assertion holds when $i=2$. 
We assume that $i=1$.
By Lemma \ref{push}, we may assume that $X$ is minimal.

\begin{itemize}
\item{The case where $\kappa(X, K_X)=-\infty$.}
If $X\cong \PP^2_k$, then the assertion follows from Proposition \ref{BVonSRC}. Thus we assume that $X$ has a ruled surface structure
$f \colon X \to C$.
Since $C$ is $F$-split by Remark \ref{gFreg} (2), it follows that $\dim\, H^1(C, \sO_C)=0$ or $1$ and the assertion follows from Lemma \ref{ruled}.

\item{The case where $\kappa(X, K_X)=0$.}
First, we assume that $X$ is a K3 surface. Then the Artin-Mazur height of $X$ is equal to one by \cite[5.1 Theorem]{GK} and hence we obtain the assertion by Proposition \ref{not-uniruled} and Remark \ref{rem-not-uniruled}.

Next, we assume that $X$ is an Enriques surface. 
We first assume that $p\neq 2$. Then there exists a finite \'etale morphism $f\colon Y\to X$ from a K3 surface $Y$. Since $f$ is \'etale, it follows from \cite[Lemma 2.5.5 (d)]{AWZ} that $Y$ is $F$-split and in particular $Y$ is not uniruled. The \'etaleness of $f$ also shows that $X$ is not uniruled and we obtain the assertion by Proposition \ref{not-uniruled}.
We next assume that $p=2$.
Since $X$ is $F$-split, the Frobenius action on $H^1(X, \sO_X)$ is bijective by \cite[Lemma 2.5.5 (a)]{AWZ} and hence $X$ is not supersingular. Moreover, since there exists a section $\sigma \in  H^0(X, \sO_X(-K_X))$ which induces a splitting of the Frobenius map, it follows that $K_X$ is not torsion and hence $X$ is not classical. Thus $X$ is a singular Enriques surface and hence there exists a finite \'etale morphism $f\colon Y\to X$ from a K3 surface $Y$ by \cite[Corollary in Section 3]{BM}. Now, the same argument as in the case where $p\neq 2$ shows the assertion.

If $X$ is an Abelian surface, then the assertion follows immediately from Proposition \ref{not-uniruled}.
Finally, we assume that $X$ is a (quasi-)hyperelliptic surface.
Since $X$ is $F$-split, a general fiber of the Albanese map is normal by \cite[Proposition 7.2]{Eji19}. Thus $X$ is a hyperelliptic surface and there exists a finite \'etale morphism $f\colon Y\to X$ from an Abelian surface $Y$. Then $X$ is not uniruled and we conclude the assertion by Proposition \ref{not-uniruled}.
\end{itemize}
\end{proof}
%%%%%%%%%%%%%%%%%%%%%%%%%%%%%%%%%%%%%%%%%%%%%%%%%%%%%%%%%%%%%%%%%%%%%%
%%%%%%%%%%%%%%%%%%%%%%%%%%%%%%%%%%%%%%%%%%%%%%%%%%%%%%%%%%%%%%%%%%%%%%
\section{Bogomolov-Sommese type vanishing for globally $F$-regular threefolds}\label{Sec:GFR}
%%%%%%%%%%%%%%%%%%%%%%%%%%%%%%%%%%%%%%%%%%%%%%%%%%%%%%%%%%%%%%%%%%%%%%
%%%%%%%%%%%%%%%%%%%%%%%%%%%%%%%%%%%%%%%%%%%%%%%%%%%%%%%%%%%%%%%%%%%%%%
In this section, we prove a Bogomolov-Sommese type vanishing on globally $F$-regular threefolds. 

\begin{defn}
Let $X$ be a variety and $\F$ a coherent sheaf on $X$.
We say that $\F$ satisfies \textit{Serre's condition }$S_n$ if $\mathrm{depth}_{\sO_{X,x}}(\F_{x})\geq \min\{n, \dim\,\sO_{X,x}\}$ holds for every (not necessary closed) point $x\in X$.
\end{defn}

\begin{lem}\label{amplevani}
Let $X$ be a projective variety and $A$ an ample Cartier divisor on $X$.
Let $\F$ be a coherent sheaf on $X$ satisfying Serre's condition $S_n$.
Then \[H^{i}(X, \F \otimes \sO_X(-mA))=0\]  for all $i<l\coloneqq \min\{n, \dim X\}$ and all sufficiently large $m$.
\end{lem}
\begin{proof}
We may assume that $X$ is the closed subscheme of $\PP_{k}^N$ and $\sO_X(A)=\sO_X(1)$.
We fix a closed point $x \in X$. Since $\F$ satisfies Serre's condition $S_n$, it follows that 
\[
\mathrm{pd}_{\sO_{\PP_{k}^N,x}}(\F_{x})=N-\mathrm{depth}_{\sO_{\PP_{k}^N,x}}(\F_{x})=N-\mathrm{depth}_{\sO_{X,x}}(\F_{x})\leq N-l
\]
and hence $\mathcal{E}xt^j_{\PP_{k}^N}(\F, \,-)=0$ for $j>N-l$.
Now the Serre duality yields 
\begin{align*}
H^i(X, \F(-m))\cong&Ext^{N-i}(\F, \omega_{\PP_{k}^N}(m))& \\
                   \cong&H^0(\PP_{k}^N, \mathcal{E}xt^{N-i}_{\PP_{k}^N}(\F, \omega_{\PP_{k}^N}(m))) &\text{$m\gg0$} \\
                   =&0 &\text{$i<l$} 
\end{align*}
and hence we obtain the assertion.
\end{proof}

\begin{lem}\label{h^0vani}
Let $X$ be a normal projective variety and $D$ a $\Q$-Cartier Weil divisor on $X$ such that $\kappa(X, D)>0$.
Let $\F$ be a reflexive sheaf on $X$.
Then $H^0(X, (\F \otimes \sO_X(-p^eD))^{**})=0$ for all sufficiently large and divisible $e$.
\end{lem}
\begin{proof}
Since $\kappa(X, D)>0$, there exist $m,n \in \Z_{>0}$ and a rational map $\phi\coloneqq\phi_{|p^m(p^n-1)D|} \colon X\dasharrow Y$ such that $Y$ is a projective variety with $\dim\, Y>0$. We fix such $m, n$.
Since $\F$ is reflexive, we can take an open subset $U$ with $\codim_X(X-U)\geq 2$ such that $\F$ is locally free on $U$ and $U\subset X_{\reg}$.
By taking a resolution of indeterminacy of $\phi|_U$, we have the following commutative diagram
\begin{equation*}
\xymatrix{ V \ar[d]_{f} \ar[rd]^{g} &\\
                 U \ar@{.>}[r]_{\phi|_U}   & Y.}
\end{equation*}
We note that $p^m(p^n-1)f^{*}D|_U-g^{*}H\geq 0$ for some ample Cartier divisor $H$ on $Y$ by the construction of $g$.
Then we have
\[
\begin{array}{rl}
H^0(X, (\F\otimes \sO_X(-p^{m+ln}D))^{**})=&H^0(U, (\F|_U\otimes \sO_U(-p^mD))\otimes \sO_U(-p^m(p^{ln}-1)D))\\
                                          =&H^0(V, f^{*}(\F|_U\otimes \sO_U(-p^mD))\otimes \sO_V(-p^m(p^{ln}-1)f^{*}D))\\
                                          \hookrightarrow&H^0(V, f^{*}(\F|_U\otimes \sO_U(-p^mD))\otimes \sO_V(-g^{*}H_l))\\
                                          =&H^0(Y, g_{*}f^{*}(\F|_U\otimes \sO_U(-p^mD))\otimes \sO_Y(-H_l))\\
\end{array}
\]
for all $l \in \Z_{>0}$, where $H_l\coloneqq(1+p^n+\cdots +p^{(l-1)n})H$.
Since $\F|_U$ and $\sO_U(-p^mD)$ are locally free, it follows that $g_{*}f^{*}(\F|_U\otimes \sO_U(-p^mD))$ is torsion-free.
Therefore $H^0(Y, g_{*}f^{*}(\F|_U\otimes \sO_U(-p^mD))\otimes \sO_Y(-H_l))=0$ for a sufficiently large integer $l$ by Lemma \ref{amplevani}.
\end{proof}

\begin{eg}
Let $X$ be a normal projective variety which lifts to the ring of Witt vectors of length two $W_2(k)$ with its Frobenius morphism (see \cite[Section 2]{BTLM} for more details).
Then there exists a splitting injective map $\Omega_X^{[i]}\hookrightarrow F_{*}\Omega_X^{[i]}$ by \cite[Theorem 2]{BTLM}. 
Let $D$ be a $\Q$-Cartier Weil divisor on $X$. 
If $\kappa(X, D)>0$, then it follows that
\[
H^0(X, (\Omega_X^{[i]}\otimes \sO_X(-D))^{**})\hookrightarrow H^0(X, (\Omega_X^{[i]}\otimes \sO_X(-p^eD))^{**})=0
\]
for a sufficiently large and divisible integer $e$ by Lemma \ref{h^0vani}. 
We remark that toric varieties lift to $W_2(k)$ with their Frobenius morphisms, but a stronger assertion than the above holds on them. We refer to  \cite[Theorem 2.22]{Fujino} for the detail.
\end{eg}

\begin{thm}\label{key}
Let $X$ be a projective globally $F$-regular variety and $\Delta$ a Weil divisor on $X$.
Suppose that $\dim X \geq 2$ and $\codim_X((X, \Delta)_{\nsnc})\geq 3$. 
Then $H^0(X, (\Omega_X^{[1]}(\log \, \Delta)\otimes \sO_X(-D))^{**})=0$ for every nef and big $\Q$-Cartier Weil divisor $D$ on $X$.
\end{thm}

\begin{proof}
First, we show the following claim.
\begin{cl}
$H^1(U, \sO_U(-D))=0$ for every nef and big $\Q$-Cartier Weil divisor $D$, where $U$ denotes $(X, \Delta)_{\snc}$.
\end{cl}
\begin{proof}[Proof of the Claim]
Let $D$ be a nef and big $\Q$-Cartier Weil divisor. We fix $m, n\in \Z_{>0}$ such that $D'\coloneqq p^m(p^n-1)D$ is Cartier.
The bigness of $D$ shows that $p^mD$ is linearly equivalent to an effective Weil divisor for all sufficiently large $m\gg0$.
By Remark \ref{gFreg} (1), there exists $l\gg 0$ such that 
\[
\sO_X\hookrightarrow F^{m+ln}_* \mathcal O_X(p^mD)
\]
splits.
By restricting to $U$ and tensoring $\sO_U(-D)$, we have a splitting injective map
\[
\begin{array}{rl}
\sO_U(-D )\hookrightarrow& F^{m+ln}_*\mathcal O_U(p^mD-p^{m+ln}D)\\
=& F^{m+ln}_* \sO_U(-D'_l),\\
\end{array}
\]
where $D'_l\coloneqq(1+p^n+\cdots +p^{(l-1)n})D'$.
By taking the cohomology, we have a splitting injection
\[
\begin{array}{rl}
H^1(U, \sO_U(-D))\hookrightarrow H^1(U, \sO_U(-D'_l)),
\end{array}
\]
and thus we may assume that $D$ is Cartier.
If $\dim \, X=2$, then $U=X$ by the assumption that $\codim_X((X, \Delta)_{\nsnc})\geq 3$, and the assertion of the claim follows from \cite[Corollary 4.4]{Smi}.
Now we assume that $n\coloneqq\dim \, X\geq 3$.
Since $X$ is globally $F$-regular, it follows that $X$ is Cohen-Macaulay by Remark \ref{gFreg} (3) and the line bundle $\sO_X(-D)$ satisfies Serre's condition $S_n$.
By the assumption $\codim_{X}(Z)\geq 3$ and by \cite[Proposition 1.2]{BH93}, we have $\mathcal{H}_Z^j(X, \sO_X(-D))=0$ for all $j<3$, 
where $Z$ denotes $(X, \Delta)_{\nsnc}$.
We consider the spectral sequence
\[
\begin{array}{rl}
E_2^{i, j}=H^i(X, \mathcal{H}_Z^j(X, \sO_X(-D)) \to H^{i+j}=H^{i+j}_{Z}(X, \sO_X(-D)).
\end{array}
\]
Since $E_2^{i, j}=H^i(X, \mathcal{H}_Z^j(X, \sO_X(-D))=0$ for all $j<3$,
we have
\[
H_{Z}^{i}(X, \sO(-D))=H^{i}=E_2^{i, 0}=H^i(X, \mathcal{H}_Z^0(X, \sO_X(-D))=0
\]
for all $i<3$.
By the local cohomology exact sequence, we have the exact sequence
\[
\begin{array}{rl}
H^1(X, \sO_X(-D)) \to H^1(U, \sO_X(-D) \to H_{Z}^{2}(X, \sO_X(-D))=0.
\end{array}
\]
Therefore it suffices to show that $H^1(X, \sO_X(-D))=0$ and this follows from  \cite[Corollary 4.4]{Smi}.
\end{proof}
%%%%%%%%%%%%%%%%%%%%%%%%%%%%%%%%%%
Now, we show the assertion of the theorem.
Conversely, we assume that
\[
H^0(X, (\Omega_X^{[1]}(\log \, \Delta) \otimes \sO_X(-D))^{**})=H^0(U, \Omega_U(\log \, \Delta)\otimes \sO_U(-D))\neq0.
\]
Then, by Lemma \ref{h^0vani}, there exists $l \in \Z_{\geq0}$ such that 
\[
\begin{array}{rl}
&H^0(X, (\Omega_X^{[1]}(\log \, \Delta) \otimes \sO_X(-p^{l}D))^{**})\neq 0,\\
&H^0(X, (\Omega_X^{[1]}(\log \, \Delta) \otimes \sO_X(-p^{l+1}D))^{**})= 0.
\end{array}
\]
We note that we can use the Cartier operator on $U=(X, \Delta)_{\snc}$.
Since $X$ is $F$-split, the exact sequence
\[
\begin{array}{rl}
0 \to \sO_U \to F_{*}\sO_U \to B_U^1 \to 0
\end{array}
\]
splits.
By the claim and the splitting of the above exact sequence, we have $H^1(U, B_U^1\otimes \sO_U(-p^{l}D))=0$ for every nef and big $\Q$-Cartier Weil divisor $D$.
Since $Z_U^1(\log \, \Delta)\subset F_{*}\Omega_U(\log \, \Delta)$, we have
\[
\begin{array}{rl}
H^0(U, Z_U^1(\log \, \Delta)\otimes \sO_U(-p^{l}D))\hookrightarrow &H^0(U, \Omega_U(\log \, \Delta)\otimes \sO_U(-p^{l+1}D))\\
                                                                 =&H^0(X, (\Omega_X^{[1]}(\log \, \Delta) \otimes \sO_X(-p^{l+1}D))^{**})\\
                                                                 =&0.\\
\end{array}
\]
Now, since $H^0(U, Z_U^1(\log \, \Delta)\otimes \sO_U(-p^{l}D))=H^1(U, B_U^1\otimes \sO_U(-p^{l}D))=0$, the exact sequence
\[
\begin{array}{rl}
0 \to B_U^1(\log \, \Delta)=B_U^1 \to  Z_U^1(\log \, \Delta) \to \Omega_U(\log \, \Delta) \to 0
\end{array}
\]
shows 
\[
\begin{array}{rl}
H^0(X, (\Omega_X^{[1]}(\log \, \Delta) \otimes \sO_X(-p^{l}D))^{**})=H^0(U, \Omega_U(\log \, \Delta)\otimes \sO_U(-p^{l}D))=0,
\end{array}
\]
a contradiction with the assumption of $l$.
\end{proof}

If $X$ is smooth in Theorem \ref{key}, then we can weaken the assumption that $X$ is globally $F$-regular as follows. 

\begin{prop}\label{ANvani}
Let $X$ be a smooth projective $F$-split variety of $\dim X\geq 2$ and $\Delta$ a simple normal crossing divisor on $X$.
Then $H^0(X, \Omega_X(\log \, \Delta) \otimes \sO_X(-D))=0$ for every nef and big Cartier divisor $D$ on $X$.
\end{prop}
\begin{proof}
We take a nef and big Cartier divisor $D$.
Since $X$ is $F$-split, we have a splitting injective map
\[
\begin{array}{rl}
H^1(X, \sO_X(-D))\hookrightarrow H^1(X, \sO_X(-p^eD)).
\end{array}
\]
By \cite[Proposition 2.24]{Lan09}, we have $H^1(X, \sO_X(-p^eD))=0$ for a sufficiently large integer $e$ and hence $H^1(X, \sO_X(-D))=0$.
Now the argument after the claim of Theorem \ref{key} shows the assertion.
\end{proof}

Globally $F$-regular surfaces have only $F$-regular singularities.
We note that $F$-regular singularities are klt and in particular the minimal resolutions are log resolutions. We refer to \cite{Hara2} for more details. 

Graf \cite{Gra} shows that a surface with $F$-regular singularities satisfies the extension theorem for the logarithmic differential form.

\begin{thm}[\textup{cf.~\cite[Theorem 1.2]{Gra}}]\label{RET}
Let $X$ be a normal surface with $F$-regular singularities and $\pi \colon Y \to X$ the minimal resolution with the reduce $\pi$-exceptional divisor $E$. 
Then $\pi_{*}\Omega_Y(\log \, E) \cong \Omega^{[1]}_X$.
\end{thm}
\begin{proof}
We may assume that $k$ is an algebraically closed field by \cite[Proposition 7.4]{Gra}.
By \cite[Theorem 1.1]{Hara2}, the dual graph of $\pi$ is one of the following.
\begin{enumerate}
\item {The graphs of the singularity is a chain.}
\item {The graphs of the singularity is star-shaped and either} 
\begin{enumerate}
\item{of type $(2, 2, d)$, $d\geq2$, and $p\neq2$,}
\item{of type $(2, 3, 3)$ or $(2, 3, 4)$, and $p>3$,}
\item{of type $(2, 3, 5)$ and $p>5$.}
\end{enumerate}
\end{enumerate}
By applying \cite[7.B Proof of Theorem 1.2 (7.9.5)]{Gra} (resp.~[\textit{ibid}, (7.9.6)], [\textit{ibid}, (7.9.7)]) to (1) (resp.~(2)(a), (2)(b) and (c)), we obtain the assertion. 
 \end{proof}

\begin{cor}\label{BVonGFSS}
Let $X$ be a normal projective $F$-split surface with $F$-regular singularities.
Then $H^0(X, \Omega^{[i]}_X\otimes \sO_X(-D))=0$ for every $i\geq0$ and every nef and big Cartier divisor $D$ on $X$.
\end{cor}
\begin{proof}
When $i=0$, we obtain the assertion by the bigness of $D$.
Since $X$ is $F$-split, it follows that $-K_X$ is effective and hence the assertion holds when $i=2$.
Now, we assume that $i=1$.
Conversely, we assume that there exists an injective homomorphism $\sO_X(D) \hookrightarrow \Omega_X^{[1]}$ for some nef and big Cartier divisor $D$ on $X$.
Let $\pi \colon Y \to X$ be the minimal resolution with the reduced $\pi$-exceptional divisor $E$. Since $\pi$ is crepant, it follows that $Y$ is $F$-split by \cite[1.3.13 Lemma]{fbook}.
Now, by Theorem \ref{RET}, we have an injective homomorphism $\sO_Y(\pi^*D) \to \pi^{*}\Omega_X^{[1]}\cong\pi^{*}\pi_{*}\Omega_Y(\log \, E) \to \Omega_Y(\log E)$, a contradiction with Proposition \ref{ANvani}.
\end{proof}

\begin{lem}\label{BVonGFR2}
Let $f \colon X \to Y$ be a projective surjective morphism of normal varieties satisfying $f_{*}\sO_X=\sO_Y$.
Suppose that a general fiber $F$ of $f$ is globally $F$-regular and $\dim\,F=1$ or $2$.
In addition, assume that $\codim_X(X_{\sg})\geq3$ when $\dim\,F=2$.
Let $D$ be an $f$-nef and $f$-big $\Q$-Cartier Weil divisor on $X$.
Then $f_{*}(\Omega_X^{[i]} \otimes \sO_X(-D))^{**}=0$ for all $i\geq0$. 
\end{lem}

\begin{proof}
If $\dim \, Y=0$, then $X$ is a smooth rational curve or a smooth rational surface, and the assertion follows from Proposition \ref{BVonSRC}.
Thus we assume that $\dim \, Y>0$. 
We may assume that $Y$ is affine.
Conversely, we assume that there exists an injective homomorphism $s\colon \sO_X(D)\hookrightarrow \Omega_X^{[i]}$ for some $i\geq 0$. 
Since the closed point $y\coloneqq f(F)$ is contained in $Y_{\reg}$, we have $I_{F}/I_{F}^2=f^{*}(\mathrm{m}_{y}/\mathrm{m}_{y}^2)=\sO_F^{\oplus\dim\,Y}$, where $I_F$ is the ideal sheaf of $F$.
Now by the conormal exact sequence, we have
\[
0 \rightarrow \sO_{F_{\reg}}^{\oplus\dim\,Y}\rightarrow \Omega_X|_{F_{\reg}}\rightarrow \Omega_{F_{\reg}} \rightarrow 0.
\]
By the generality of $F$ and the assumption of the lemma, the fiber $F$ is contained in $X_{\reg}$ and the restriction $s|_F \colon \sO_F(D|_F)\hookrightarrow \Omega^{[i]}_X|_{F}=\bigwedge^{i}\Omega_X|_{F}$ is injective. Then we obtain an injective homomorphism $\sO_F(D|_F)\hookrightarrow \Omega^{[j]}_{F}\otimes \bigwedge^{i-j}\sO_F^{\oplus\dim\,Y}$
by \cite[Lemma 3.14]{Gra15}.
In particular, we have $\sO_F(D|_F) \hookrightarrow \Omega^{[j]}_{F}$ for some $j\geq 0$. 
We first assume that $\dim\,F=1$. Then $F\cong \PP^1_k$ and $D|_F$ is a nef and big Cartier divisor. This is a contradiction.
We next assume that $\dim\,F=2$.
Then $D|_F$ is a nef and big Cartier divisor by the assumption that $\codim_X(X_{\sg})\geq 3$. Now we can derive a contradiction by Corollary \ref{BVonGFSS}. 
\end{proof}

Let us recall the definition of Mori fiber spaces.

\begin{defn}\label{MFS}
Let $f \colon X \to Y$ be a projective surjective morphism of normal varieties.
We say $f \colon X \to Y$ is a \emph{Mori fiber space} if
\begin{itemize}
    \item $-K_X$ is $f$-ample,
    \item $X$ is $\Q$-factorial,
    \item $f_{*}\sO_X=\sO_Y$ and $\dim \, X>\dim \, Y$,
    \item the relative Picard rank $\rho(X/Y)=1$.
\end{itemize}
\end{defn}

Now, we prove a Bogomolov-Sommese type vanishing for globally $F$-regular threefolds. 

\begin{thm}\label{main}
Let $X$ be a smooth projective globally $F$-regular threefold and $\sO_X(D) \subset \Omega_X$ an invertible subsheaf. 
If $p>3$, then $\kappa(X, D) \leq 1$.
Furthermore, if $p>7$, then $\kappa(X, D) \leq 0$.
\end{thm}
\begin{rem}
In the above theorem, we need the assumption that $p>3$ only for running $K_X$-MMP.
\end{rem}
\begin{proof}
Let us prove the first assertion of the theorem. We assume that $p>3$ and $\kappa(X, D)>1$. Let us show that $H^0(X, \Omega_{X}\otimes \sO_X(-D))=0$.
Since $X$ is globally $F$-regular, the anti-canonical divisor $-K_X$ is big by \cite[Corollary 4.5]{SS10}. Then by running $K_X$-MMP, we obtain a birational contraction $f\colon X\dasharrow X'$ and a Mori fiber space $g\colon X'\to Y$ by \cite[Theorem 1.2]{HW19}.
By Remark \ref{gFreg} (2), $X'$ is a $\Q$-factorial terminal projective globally $F$-regular threefold.
By Lemma \ref{push}, it suffices to show that
$H^0(X', (\Omega^{[1]}_{X'}\otimes \sO_{X'}(-D'))^{**})=0$. Moreover, we have $\kappa(X',D')\geq \kappa(X,D)> 1$.

First, we assume that $\dim\, Y=0$.
In this case, the divisor $D'$ is ample since $\kappa(X', D')> 1$ and $\rho(X')=1$.
Since three-dimensional terminal singularities are isolated by \cite[Corollary 2.13]{Kol13}, we obtain $H^0(X', (\Omega^{[1]}_{X'}\otimes \sO_{X'}(-D'))^{**})=0$ by Theorem \ref{key}.

Next, we assume that $\dim\, Y=1$.
Let $G$ be a general fiber of $g$. 
Since $-K_{X'}$ and $G$ form the basis of $N^1(X')\otimes_{\Z}{\Q}$, we can denote $D'\equiv a(-K_{X'})+bG$ for some $a,b \in \Q$, where $N^1(X')$ is the quotient of $\Pic(X')$ by its subgroup consisting of all isomorphism classes numerically equivalent to zero.
We denote by $\Pic^{0}(X')$ the subgroup of $\Pic(X')$ consisting of all isomorphism classes algebraically equivalent to zero and by $\NS(X')$ the quotient of $\Pic(X')$ by $\Pic^{0}(X')$. 
Since $X'$ is globally $F$-regular, it follows from \cite[Corollary 4.3]{Smi} that $H^1(X', \sO_{X'})=0$.
Then by \cite[Theorem 9.5.11]{FAG}, we obtain $\Pic^{0}(X')=0$ and hence $\Pic(X')=\NS(X')$. 
Since the kernel of the natural map $\NS(X')\twoheadrightarrow N^1(X')$ is torsion by \cite[Corollary 1.4.38]{Lazarsfeld}, we obtain $\Pic(X')\otimes_{\Z}\Q=\NS(X')\otimes_{\Z}\Q=N^1(X')\otimes_{\Z}\Q$.
In particular, $D'$ is $\Q$-linearly equivalent to $a(-K_{X'})+bG$.
Since $\kappa(X', D')>1=\kappa(X', G)$, it follows that $a>0$ and hence $D'|_G$ is ample. 
Now $G$ is a globally $F$-regular surface by Theorem \ref{fiber} and hence we obtain $H^0(X', (\Omega^{[1]}_{X'}\otimes \sO_{X'}(-D'))^{**})=0$ by Lemma \ref{BVonGFR2}.

Finally, we assume that $\dim\, Y=2$.
In this case, $X'$ is separably rationally connected by Theorem \ref{glsrc} (1) and hence so is $X$.
Then we obtain $H^0(X, \Omega_{X}\otimes \sO_X(-D))=0$ by Proposition \ref{BVonSRC}.

Now, we show the latter assertion. We assume that $p>7$ and $\kappa(X,D)>0$.
We take $X',Y$ as above.
When $\dim\,Y=0$ or $2$, we obtain the assertion by the essentially same argument as above.
When $\dim\,Y=1$, Theorem \ref{glsrc} (2) shows that $X'$ is separably rationally connected and hence so is $X$. Therefore we obtain $H^0(X, \Omega_{X}\otimes \sO_X(-D))=0$ by Proposition \ref{BVonSRC}.
\end{proof}

\begin{rem}\label{remmain}
Let $X$ be a terminal projective globally $F$-regular threefold. Suppose that $p>3$ and $\sO_X(D) \subset \Omega^{[1]}_X$ is a Weil divisorial subsheaf.
Then an argument similar to Theorem \ref{main} shows that $\kappa(X, D) \leq 2$ as follows.

By taking a small $\Q$-factorialization and running $K_{X}$-MMP, the assertion is reduced to a Mori fiber space $g \colon X' \to Y$. 
Let $D'$ is the push-forward of $D$ to $X'$.
When $\dim\,Y=0$ or $1$, we obtain the assertion by the proof of the first assertion of Theorem \ref{main}.
On the other hand, when $\dim\,Y=2$, we need a different argument from Theorem \ref{main} since Proposition \ref{BVonSRC} can not be applied to singular varieties. 
In this case, since $D'$ is big and $\rho(X'/Y)=1$, it follows that $D'|_G$ is ample and the assertion follows from Lemma \ref{BVonGFR2}.
\end{rem}

%%%%%%%%%%%%%%%%%%%%%%%%%%%%%%%%%%%%%%%%%%%%%%%%%%%%%%%%%%%%%%%%%%%%%%
%%%%%%%%%%%%%%%%%%%%%%%%%%%%%%%%%%%%%%%%%%%%%%%%%%%%%%%%%%%%%%%%%%%%%%

\section*{Acknowledgement}
The author wishes to express his gratitude to his supervisor Professor Shunsuke Takagi for his encouragement, valuable advice, and suggestions. 
He would like to thank Professor Adrian Langer for pointing out some references, Professor Hiromu Tanaka, Kenta Sato, and Shou Yoshikawa for helpful comments and conversations. This work was supported by JSPS KAKENHI 19J21085.

%%%%%%%%%%%%%%%%%%%%%%%%%%%%%%%%%%%%%%%%%%%%%%%%%%%%%%%%%%%%%%%%%%%%%%
%%%%%%%%%%%%%%%%%%%%%%%%%%%%%%%%%%%%%%%%%%%%%%%%%%%%%%%%%%%%%%%%%%%%%%
%%%%%%%%%%%%%%%%%%%%%%%%%%%%%%%%%%%%%%%%%%%%%%%%%%%%%%%%%%%%%%%%%%%%%%
%%%%%%%%%%%%%%%%%%%%%%%%%%%%%%%%%%%%%%%%%%%%%%%%%%%%%%%%%%%%%%%%%%%%%%

%\bibliography{hoge.bib}
%\bibliographystyle{abbrv}
%\bibliographystyle{alpha}

\end{document}